%% file: Martingales-ecp-v2.tex
\documentclass[11pt,letterpaper,twoside,reqno]{amsart}

\input{macro-file}

\usepackage{subfigure}
\usepackage{wasysym}

\begin{document}
\title[The Matrix Freedman Inequality]
{Freedman's Inequality for Matrix Martingales}

\author{Joel A.~Tropp}

\keywords{Discrete-time martingale, large deviation, probability inequality, random matrix, sum of independent random variables}

\thanks{2010 {\it Mathematics Subject Classification}.
Primary:
60B20. %(probabilistic aspects of random matrices).
Secondary:
60F10, %(large deviations)
60G50, %(sums of independent random variables)
60G42%%%(discrete martingales) 
}

\thanks{JAT is with Computing \& Mathematical Sciences, MC 305-16, California Inst.~Technology, Pasadena, CA 91125.
E-mail: \url{jtropp@acm.caltech.edu}.
Research supported by ONR award N00014-08-1-0883, DARPA award N66001-08-1-2065, and AFOSR award FA9550-09-1-0643.}

\date{15 June 2010.  Revised 14 November 2010 and 15 January 2011.}

\begin{abstract}
Freedman's inequality is a martingale counterpart to Bernstein's inequality.  This result shows that the large-deviation behavior of a martingale is controlled by the predictable quadratic variation and a uniform upper bound for the martingale difference sequence.  Oliveira has recently established a natural extension of Freedman's inequality that provides tail bounds for the maximum singular value of a matrix-valued martingale.  This note describes a different proof of the matrix Freedman inequality that depends on a deep theorem of Lieb from matrix analysis.  This argument delivers sharp constants in the matrix Freedman inequality, and it also yields tail bounds for other types of matrix martingales.  The new techniques are adapted from recent work~\cite{Tro10:User-Friendly-arxiv} by the present author.
\end{abstract}

%The main argument begins with a matrix extension of the Laplace transform method due to Ahlswede and Winter.  This paper develops a powerful new framework for controlling the matrix moment generating function by invoking a deep theorem of Lieb.  This technique replaces earlier methods based on the Golden--Thompson inequality that lead to fundamentally weaker results.  The martingale theory depends on Oliveira's extension of Freedman's techniques, but Lieb's theorem again yields substantial improvements over previous work.

%  The central new idea is that a deep inequality of Lieb provides superior bounds for the matrix moment generating function.  This technique replaces earlier methods based on the Golden--Thompson inequality that produce fundamentally weaker results.  

\maketitle

\section{An Introduction to Freedman's Inequality} \label{sec:freedman}

The Freedman inequality~\cite[Thm.~(1.6)]{Fre75:Tail-Probabilities}
is a martingale extension of the Bernstein inequality.  This result demonstrates that a martingale exhibits normal-type concentration near its mean value on a scale determined by the predictable quadratic variation, and the upper tail has Poisson-type decay on a scale determined by a uniform bound on the difference sequence.

Oliveira~\cite[Thm.~1.2]{Oli10:Concentration-Adjacency} proves that Freedman's inequality extends, in a certain form, to the matrix setting.  The purpose of this note is to demonstrate that the methods from the author's paper~\cite{Tro10:User-Friendly-arxiv} can be used to establish a sharper version of the matrix Freedman inequality.  Furthermore, this approach offers a transparent way to obtain other probability inequalities for adapted sequences.

Let us introduce some notation and background on martingales so that we can state Freedman's original result rigorously.  Afterward, we continue with a statement of our main results and a presentation of the methods that we need to prove the matrix generalization.

\subsection{Martingales}

Let $(\Omega, \coll{F}, \mathbb{P})$ be a probability space, and let $\coll{F}_0 \subset \coll{F}_1 \subset \coll{F}_2 \subset \dots \subset \coll{F}$ be a filtration of the master sigma algebra.  We write $\Expect_k$ for the expectation conditioned on $\coll{F}_k$.  A \term{martingale} is a (real-valued) random process $\{ Y_k : k = 0, 1, 2, \dots \}$ that is adapted to the filtration and that satisfies two properties:
$$
\Expect_{k-1} Y_k = Y_{k-1}
\quad\text{and}\quad
\Expect \abs{Y_k} < + \infty
\quad\text{for $k = 1, 2, 3, \dots$.}
$$
For simplicity, we assume that the initial value of a martingale is null: $Y_0 = 0$.  The \term{difference sequence} is the random process defined by
$$
X_k = Y_k - Y_{k-1}
\quad\text{for $k = 1, 2, 3, \dots$.}
$$
Roughly, the present value of a martingale depends only on the past values, and the martingale has the status quo property: today, on average, is the same as yesterday.

\subsection{Freedman's Inequality}

Freedman uses a powerful stopping-time argument to establish the following theorem for scalar martingales~\cite[Thm.~(1.6)]{Fre75:Tail-Probabilities}.

\begin{thm}[Freedman] \label{thm:freedman}
Consider a real-valued martingale $\{ Y_k : k = 0, 1, 2, \dots \}$ with difference sequence $\{ X_k : k = 1, 2, 3, \dots \}$.  Assume that the difference sequence is uniformly bounded:
$$
X_k \leq R
\quad\text{almost surely}
\quad\text{for $k = 1, 2, 3, \dots$.}
$$
Define the predictable quadratic variation process of the martingale:
$$
W_k := \sum\nolimits_{j=1}^k \Expect_{j-1} \big(X_j^2\big)
\quad\text{for $k = 1, 2, 3, \dots$.}
$$
Then, for all $t \geq 0$ and $\sigma^2 > 0$,
$$
\Prob{ \exists k \geq 0 : Y_k \geq t \ \text{ and }\ 
	W_k \leq \sigma^2 }
	\leq \exp \left\{ - \frac{ -t^2/2 }{\sigma^2 + Rt/3} \right\}.
$$
\end{thm}

When the difference sequence $\{ X_k \}$ consists of independent random variables, the predictable quadratic variation is no longer random.  In this case, Freedman's inequality reduces to the usual Bernstein inequality~\cite[Thm.~6]{Lug09:Concentration-Measure}.

\subsection{Matrix Martingales}

Matrix martingales are defined in much the same manner as scalar martingales.  Consider a random process $\{ \mtx{Y}_k : k = 0, 1, 2, \dots \}$ whose values are matrices of finite dimension.  We say that the process is a \term{matrix martingale} when
$$
\Expect_{k-1} \mtx{Y}_k = \mtx{Y}_{k-1}
\quad\text{and}\quad
\Expect \norm{\mtx{Y}_k} < + \infty
\quad\text{for $k = 1, 2, 3, \dots$.}
$$
We write $\norm{\cdot}$ for the \term{spectral norm}, which coincides with the operator norm between Hilbert spaces.  As before, we assume that $\mtx{Y}_0 = \mtx{0}$, and we define the difference sequence $\{ \mtx{X}_k : k = 1, 2, 3, \dots \}$ via the relation
$$
\mtx{X}_k = \mtx{Y}_k - \mtx{Y}_{k-1}
\quad\text{for $k = 1, 2, 3, \dots$.}
$$
A matrix-valued random process is a martingale if and only if we obtain a scalar martingale when we track each fixed coordinate in time.

\subsection{Freedman's Inequality for Matrices}

In the elegant paper~\cite{Oli10:Concentration-Adjacency}, Oliveira establishes that it is possible to extend Freedman's inequality to the matrix setting.  He studies martingales that take self-adjoint matrix values, and he shows that the \emph{maximum eigenvalue} of the martingale satisfies a result very similar to Freedman's inequality.  The uniform bound $R$ and the predictable quadratic variation $\{W_k\}$ are replaced by natural noncommutative extensions.  As a consequence, these results have powerful applications in random matrix theory.

In this note, we establish a sharper version of Oliveira's theorem~\cite[Thm.~1.2]{Oli10:Concentration-Adjacency}.

\begin{thm}[Matrix Freedman] \label{thm:matrix-freedman}
Consider a matrix martingale $\{ \mtx{Y}_k : k = 0, 1, 2, \dots \}$ whose values are self-adjoint matrices with dimension $d$, and let $\{ \mtx{X}_k : k = 1, 2, 3, \dots \}$ be the difference sequence.  Assume that the difference sequence is uniformly bounded in the sense that
$$
\lambda_{\max}( \mtx{X}_k ) \leq R
\quad\text{almost surely}
\quad\text{for $k = 1, 2, 3, \dots$}.
$$
Define the predictable quadratic variation process of the martingale:
$$
\mtx{W}_k := \sum\nolimits_{j=1}^k \Expect_{j-1} \big(\mtx{X}_j^2\big).
\quad\text{for $k = 1, 2, 3, \dots$}.
$$
Then, for all $t \geq 0$ and $\sigma^2 > 0$,
$$
\Prob{ \exists k \geq 0 : \lambda_{\max}(\mtx{Y}_k) \geq t \ \text{ and }\ 
	\norm{\mtx{W}_k} \leq \sigma^2 }
	\leq d \cdot \exp \left\{ - \frac{ -t^2/2 }{\sigma^2 + Rt/3} \right\}.
$$
\end{thm}

Here and elsewhere, $\lambda_{\max}$ denotes the algebraically largest eigenvalue of a self-adjoint matrix, and $\norm{\cdot}$ denotes the spectral norm, which returns the largest singular value of a matrix.

Theorem~\ref{thm:matrix-freedman} offers several concrete improvements over Oliveira's original work.  His theorem~\cite[Thm.~1.2]{Oli10:Concentration-Adjacency} requires a stronger uniform bound of the form $\norm{ \mtx{X}_k } \leq R$, and the constants in his inequality are somewhat larger (but still very reasonable).

We prove Theorem~\ref{thm:matrix-freedman} in Section~\ref{sec:freedman} as a consequence of a stronger probability inequality that follows from a general result for adapted sequences of matrices.  These tail bounds cannot be sharpened without changing their structure; see~\cite[\S4 and \S6]{Tro10:User-Friendly-arxiv} for a more detailed discussion.

As an immediate corollary of Theorem~\ref{thm:matrix-freedman}, we obtain a result for rectangular matrices.

\begin{cor}[Rectangular Matrix Freedman]
Consider a matrix martingale $\{ \mtx{Y}_k : k = 0, 1, 2, \dots \}$ whose values are matrices with dimension $d_1 \times d_2$, and let $\{ \mtx{X}_k : k = 1, 2, 3, \dots \}$ be the difference sequence.  Assume that the difference sequence is uniformly bounded:
$$
\norm{ \mtx{X}_k } \leq R
\quad\text{almost surely}
\quad\text{for $k = 1, 2, 3, \dots$}.
$$
Define two predictable quadratic variation processes for this martingale:
\begin{align*}
\mtx{W}_{{\rm col}, \, k} &:= \sum\nolimits_{j=1}^k \Expect_{j-1} \big(\mtx{X}_j \mtx{X}_j^\adj \big) \quad\text{and} \\
\mtx{W}_{{\rm row}, \, k} &:= \sum\nolimits_{j=1}^k \Expect_{j-1} \big(\mtx{X}_j^\adj \mtx{X}_j \big)
\quad\text{for $k = 1, 2, 3, \dots$}.
\end{align*}
Then, for all $t \geq 0$ and $\sigma^2 > 0$,
$$
\Prob{ \exists k \geq 0 : \norm{ \mtx{Y}_k } \geq t \ \text{ and }\ 
	\max\{ \norm{ \mtx{W}_{{\rm col}, \, k} }, \norm{\mtx{W}_{{\rm row}, \, k}} \} \leq \sigma^2 }
	\leq (d_1 + d_2) \cdot \exp \left\{ - \frac{ -t^2/2 }{\sigma^2 + Rt/3} \right\}.
$$
\end{cor}

\begin{proof}[Proof Sketch]
Define a self-adjoint matrix martingale $\{\mtx{Z}_k\}$ with dimension $d = d_1 + d_2$ via
$$
\mtx{Z}_k = \begin{bmatrix} \mtx{0} & \mtx{Y}_k \\ \mtx{Y}_k^\adj & \mtx{0} \end{bmatrix}.
$$
Apply Theorem~\ref{thm:matrix-freedman} to this martingale.  See~\cite[\S2.6 and \S4.2]{Tro10:User-Friendly-arxiv} for some additional details about this type of argument.
\end{proof}

\subsection{Tools and Techniques}

In his paper~\cite{Oli10:Concentration-Adjacency}, Oliveira describes a way to transport Freedman's stopping-time argument to the matrix setting.  The main technical obstacle is to control the evolution of the moment generating function (mgf) of the matrix martingale.  Oliveira accomplishes this task using an insightful variation on a idea due to Ahlswede and Winter~\cite[App.]{AW02:Strong-Converse}.  This method, however, does not result in the sharpest bounds on the matrix mgf.

This note demonstrates that the ideas from~\cite{Tro10:User-Friendly-arxiv} allow us to obtain the sharp estimates for the mgf with minimal effort.  Our main tool is a deep theorem~\cite[Thm.~6]{Lie73:Convex-Trace} of Lieb.

\begin{thm}[Lieb, 1973] \label{thm:lieb}
Fix a self-adjoint~matrix $\mtx{H}$.  The function
$$
\mtx{A} \longmapsto \trace \exp( \mtx{H} + \log(\mtx{A}))
$$
is concave on the positive-definite~cone.
\end{thm}

\noindent
See~\cite[\S3.3]{Tro10:User-Friendly-arxiv} and~\cite{Tro10:Joint-Convexity} for some additional discussion of this result.  We apply Theorem~\ref{thm:lieb} through the following simple corollary~\cite[Cor. 3.2]{Tro10:User-Friendly-arxiv}.  We include a proof for completeness.

%This observation seems to be new. %, which seems to be novel.

\begin{cor}[Tropp, 2010] %[The Cumulant Inequality]
\label{cor:cum-ineq}
Let $\mtx{H}$ be a fixed self-adjoint matrix, and let $\mtx{X}$ be a random self-adjoint matrix.  Then
$$
\Expect \trace \exp( \mtx{H} + \mtx{X} )
	\leq \trace \exp( \mtx{H} + \log( \Expect \econst^{\mtx{X}} ) ).
$$
\end{cor}

\begin{proof}
Define the random matrix $\mtx{Y} = \econst^{\mtx{X}}$, and calculate that
$$
\Expect \trace \exp(\mtx{H} + \mtx{X})
	= \Expect \trace \exp( \mtx{H} + \log( \mtx{Y} ) )
	\leq \trace \exp( \mtx{H} + \log( \Expect \mtx{Y} ) )
	= \trace \exp( \mtx{H} + \log( \Expect \econst^{\mtx{X}} ) ).
$$
The first identity follows because the logarithm can be defined as the functional inverse of the matrix exponential.  Lieb's result, Theorem~\ref{thm:lieb}, establishes that the trace function is concave in $\mtx{Y}$, so we may invoke Jensen's inequality to draw the expectation inside the logarithm.
\end{proof}

A significant advantage of our point of view is that the proof extends in a transparent way to yield other types of probability inequalities for adapted sequence of random matrices.  We have dilated on this observation in a preliminary version of this work that is now available as a technical report~\cite{Tro10:User-Friendly-Martingale-TR}.  Here, for brevity, we focus on proving Freedman's inequality.

\section{Tail Bounds via Martingale Methods} \label{sec:martingale}

In this section, we show that Freedman's techniques extend to the matrix setting with minor (but profound) changes.  The key idea is to use Corollary~\ref{cor:cum-ineq} to control the evolution of a matrix version of the moment generating function.  This argument culminates in a rather general theorem on the large deviation behavior of an adapted sequence of random matrices.  In \S\ref{sec:freedman}, we specialize this result to obtain Freedman's inequality.

\subsection{Additional Terminology}

We say that a sequence $\{ \mtx{X}_k \}$ of random matrices is \emph{adapted} to the filtration when each $\mtx{X}_k$ is measurable with respect to $\coll{F}_k$.  Loosely speaking, an adapted sequence is one where the present depends only upon the past.
We say that a sequence $\{ \mtx{V}_k \}$ of random matrices is \emph{previsible} when each $\mtx{V}_k$ is measurable with respect to $\coll{F}_{k-1}$.  In particular, the sequence $\{ \Expect_{k-1} \mtx{X}_k \}$ of conditional expectations of an adapted sequence $\{ \mtx{X}_k \}$ is previsible.
A \term{stopping time} is a random variable $\kappa : \Omega \to \mathbb{N}_0 \cup \{ \infty\}$ that satisfies
$$
\{ \kappa \leq k \} \subset \coll{F}_k
\quad\text{for $k = 0, 1, 2, \dots, \infty$.}
$$
In words, we can determine if the stopping time has arrived from current and past experience.

\subsection{The Large Deviation Supermartingale} \label{sec:supermartingale}

Consider an adapted random process $\{ \mtx{X}_k : k = 1, 2, 3, \dots \}$ and a previsible random process $\{\mtx{V}_k : k = 1, 2, 3, \dots \}$ whose values are self-adjoing~matrices with dimension $d$.  Suppose that the two processes are connected through a relation of the form
\begin{equation} \label{eqn:cgf-bound}
\log \Expect_{k-1} \econst^{\theta \mtx{X}_k}
	\psdle g(\theta) \cdot \mtx{V}_{k}
\quad\text{for $\theta > 0$,}
%\quad\text{almost surely for $\theta \in \Theta$.}
\end{equation}
%The function $g : \Theta \to \mathbb{R}_{+}$ where $\Theta$ is a set of positive numbers
where the function $g : (0,\infty) \to [0, \infty]$.
%, and---for simplicity---we do not allow this function to depend on the index $k$.
The left-hand side should be interpreted as a conditional cumulant generating function (cgf); see~\cite[Sec.~3.1]{Tro10:User-Friendly-arxiv}.
It is convenient to introduce the partial sums of the original process and the partial sums of the conditional cgf bounds:
\begin{align*}
\mtx{Y}_0 := \mtx{0}
&\quad\text{and}\quad
\mtx{Y}_k := \sum\nolimits_{j=1}^k \mtx{X}_j. \\
\mtx{W}_{0} := \mtx{0}
&\quad\text{and}\quad
\mtx{W}_{k} := \sum\nolimits_{j=1}^{k} \mtx{V}_j.
\end{align*}
The random matrix $\mtx{W}_k$ can be viewed as a measure of the total variability of the process $\{\mtx{X}_k\}$ up to time $k$.    The partial sum $\mtx{Y}_k$ is unlikely to be large unless $\mtx{W}_k$ is also large.  

%In almost all our examples, $\{ \mtx{V}_k \}$ is a sequence of psd matrices, and so $\{\mtx{W}_k\}$ increases with respect to the semidefinite order.

%In the scalar setting, Freedman~\cite[p.~101]{Fre75:Tail-Probabilities} provides the intuition that $k$ is a nominal measure of time while $\mtx{W}_k$ is an intrinsic measure of time.

%We may suppress the parameter $\theta$ when it is not relevant to the calculations.

%\notate{Probably need to add some intuition...}

To continue, we fix the function $g$ and a positive number $\theta$.  Define a real-valued function with two self-adjoint~matrix arguments:
$$
G_{\theta}( \mtx{Y}, \mtx{W} )
	:= \trace \exp\big( \theta \mtx{Y} - g(\theta) \cdot \mtx{W} \big).
$$
We use the function $G_{\theta}$ to construct a real-valued random process.
\begin{equation} \label{eqn:Sk-general}
S_k := S_k(\theta) = G_{\theta}(\mtx{Y}_k, \mtx{W}_{k} )
%= \trace \exp\big( \theta \mtx{Y}_k - \mtx{V}_{k-1}(\theta) \big)
\quad\text{for $k = 0, 1, 2, \dots$.}
\end{equation}
This process is an evolving measure of the discrepancy between the partial sum process $\{\mtx{Y}_k\}$ and the cumulant sum process $\{\mtx{W}_k\}$.  The following lemma describes the key properties of this random sequence.  In particular, the average discrepancy decreases with time.  %Once again, the proof relies on Lieb's result, Theorem~\ref{thm:lieb}.

%  Note that $F_{\theta}$ has monotonicity properties with respect to both arguments:
%\begin{equation} \label{eqn:F-monotone}
%\mtx{A} \psdle \mtx{A}_1
%\quad\text{and}\quad \mtx{H}(\theta) \psdge \mtx{H}_1(\theta)
%\quad\Longrightarrow\quad
%F_{\theta}(\mtx{A}, \mtx{H}) \leq F_{\theta}(\mtx{A}_1, \mtx{H}_1)
%\end{equation}

\begin{lemma} \label{lem:Sk-supermartingale}
For each fixed $\theta > 0$, the random process $\{S_k(\theta) : k = 0,1, 2, \dots \}$ defined in~\eqref{eqn:Sk-general} is a positive supermartingale %(i.e., adapted process with decreasing expectation)
whose initial value $S_0 = d$.
\end{lemma}

\begin{proof}
It is easily seen that $S_k$ is positive because the exponential of a self-adjoint~matrix is positive definite, and the trace of a positive-definite matrix is positive.  We obtain the initial value from a short calculation:
$$
S_0 = \trace\exp\left(\theta \mtx{Y}_0 - g(\theta) \cdot \mtx{W}_{0} \right)
	= \trace \exp( \mtx{0} )
	= \trace \Id
	= d.
$$
To prove that the process is a supermartingale, we ascend a short chain of inequalities.
\begin{align*}
\Expect_{k-1} S_k
	&= \Expect_{k-1} \trace \exp\left(\theta \mtx{Y}_{k-1} - g(\theta) \cdot \mtx{W}_{k} + \theta \mtx{X}_k \right) \\
	&\leq \trace \exp\left(\theta \mtx{Y}_{k-1} - g(\theta) \cdot \mtx{W}_{k} + \log \Expect_{k-1} \econst^{\theta\mtx{X}_k} \right) \\
	&\leq \trace\exp\left(\theta \mtx{Y}_{k-1} - g(\theta) \cdot \mtx{W}_{k}		+ g(\theta) \cdot \mtx{V}_{k} \right) \\
	&= \trace\exp\left(\theta \mtx{Y}_{k-1} - g(\theta) \cdot \mtx{W}_{k-1} \right) \\
	&= S_{k-1}.	
\end{align*}
In the second line, we invoke Corollary~\ref{cor:cum-ineq}, conditional on $\coll{F}_{k-1}$.  This act is legal because $\mtx{Y}_{k-1}$ and $\mtx{W}_{k}$ are both measurable with respect to $\coll{F}_{k-1}$.  The next inequality depends on the assumption~\eqref{eqn:cgf-bound} together with the fact that the trace exponential is monotone with respect to the semidefinite order~\cite[\S2.2]{Pet94:Survey-Certain}.  The last step follows because $\{ \mtx{W}_{k} \}$ is the sequence of partial sums of $\{\mtx{V}_k\}$.
\end{proof}

%A critical consequence of the last lemma is the following corollary.
%
%\begin{cor} \label{cor:supermartingale}
%Let $\kappa$ be a stopping time that is finite almost surely.  Then $\Expect S_{\kappa} \leq d$.
%\end{cor}
%
%\begin{proof}
%This claim follows immediately from Doob's optional stopping theorem because $\{ S_k \}$ is a nonnegative supermartingale.  (See, for example,~\cite[Thm.~(57.4)]{RW00:Diffusions-Markov-I}.) % \notate{VERIFY!}
%\end{proof}

Finally, we present a simple inequality for the function $G_{\theta}$ that holds when we have control on the eigenvalues of its arguments.

%One imagines that stricter assumptions might lead to stronger estimates.

\begin{lemma} \label{lem:G-monotone}
Suppose that $\lambda_{\max}(\mtx{Y}) \geq t$ and that $\lambda_{\max}(\mtx{W}) \leq w$.  For each $\theta > 0$,
$$
G_{\theta}( \mtx{Y}, \mtx{W} )
	\geq \econst^{\theta t - g(\theta) \cdot w}.
$$
\end{lemma}

\begin{proof}
Recall that $g(\theta) \geq 0$.  The bound results from a straightforward calculation:
$$
G_{\theta}(\mtx{Y}, \mtx{W})
	= \trace \econst^{ \theta \mtx{Y} - g(\theta) \cdot \mtx{W} }
	\geq \trace \econst^{\theta\mtx{Y} - g(\theta) \cdot w \Id }
	\geq \lambda_{\max}\left( \econst^{ \theta \mtx{Y} - g(\theta) \cdot w \Id }\right)
	= \econst^{ \theta \lambda_{\max}( \mtx{Y} ) - g(\theta) \cdot w }
	\geq \econst^{ \theta t - g(\theta) \cdot w}.
$$
The first inequality depends on the semidefinite relation $\mtx{W} \psdle w \Id$ and the monotonicity of the trace exponential with respect to the semidefinite order~\cite[\S2.2]{Pet94:Survey-Certain}.  The second inequality relies on the fact that the trace of a psd matrix is at least as large as its maximum eigenvalue.  The third identity follows from the spectral mapping theorem and elementary properties of the maximum eigenvalue map.
\end{proof}

%\begin{rem}[Previous Work]
%Oliveira has shown that Freedman's method can be extended to the matrix setting~\cite{Oli10:Concentration-Adjacency} using the ideas of Ahlswede and Winter.  This work constructs a random process closely related to $G_\theta$, and it uses the Golden--Thompson inequality~\eqref{eqn:golden-thompson} to verify that the random process is a supermartingale.
%\end{rem}

% Another proof, possibly flawed.
%
%The bound results from a straightforward calculation:
%\begin{multline*}
%G_{\theta}(\mtx{Y}, \mtx{W})
%	= \trace \econst^{ \theta \mtx{Y} - g(\theta) \cdot \mtx{W} }
%	\geq \lambda_{\max}\left( \econst^{\theta\mtx{Y} - g(\theta) \cdot \mtx{W} } \right) \\
%	= \econst^{ \lambda_{\max}(\theta \mtx{Y} - g(\theta) \cdot \mtx{W}) }
%	\geq \econst^{ \theta \lambda_{\max}( \mtx{Y} ) - g(\theta) \lambda_{\max}(\mtx{W}) }
%	\geq \econst^{ \theta y - g(\theta) w}.
%\end{multline*}
%The first inequality relies on the property~\eqref{eqn:maxeig-trace} that the trace of a psd matrix is at least as large as its maximum eigenvalue.  The second identity is the spectral mapping theorem.  The third inequality follows from the fact that the maximum eigenvalue map satisfies the triangle inequality and is homogeneous.

\subsection{A Tail Bound for Adapted Sequences}

Our key theorem for adapted sequences provides a bound on the probability that the partial sum of a matrix-valued random process is large.  In the next section, we apply this result to establish a stronger version of Theorem~\ref{thm:matrix-freedman}.  This result also allows us to develop other types of probability inequalities for adapted sequences of random matrices; see the technical report~\cite{Tro10:User-Friendly-Martingale-TR} for additional details.

%  In the sequel, we invoke this result repeatedly to obtain a collection of probability inequalities for sums of random matrices.

\begin{thm}[Master Tail Bound for Adapted Sequences] \label{thm:main-result-adapted}
Consider an adapted sequence $\{ \mtx{X}_k \}$ and a previsible sequence $\{ \mtx{V}_k \}$ of self-adjoint matrices with dimension $d$.  Assume these sequences satisfy the relations
\begin{equation} \label{eqn:cgf-hyp-adapted}
\log \Expect_{k-1} \econst^{\theta \mtx{X}_k}
	\psdle g(\theta) \cdot \mtx{V}_{k}
\quad \text{almost surely for each $\theta > 0$},
\end{equation}
where the function $g : (0, \infty) \to [0, \infty]$.  In particular, the hypothesis~\eqref{eqn:cgf-hyp-adapted} holds when
\begin{equation} \label{eqn:mgf-hyp-adapted}
\Expect_{k-1} \econst^{\theta \mtx{X}_k}
	\psdle \econst^{g(\theta) \cdot \mtx{V}_{k}}
	\quad\text{almost surely for each $\theta > 0$.}
\end{equation}
Define the partial sum processes
$$
\mtx{Y}_k := \sum\nolimits_{j=1}^k \mtx{X}_j
\quad\text{and}\quad
\mtx{W}_{k} := \sum\nolimits_{j=1}^{k} \mtx{V}_j.
$$
Then, for all $t, w \in \mathbb{R}$,
$$
\Prob{ \exists k \geq 0 : \lambda_{\max}( \mtx{Y}_k ) \geq t \ \text{ and }\
\lambda_{\max}( \mtx{W}_{k} ) \leq w }
	\leq d \cdot \inf_{\theta > 0} \econst^{-\theta t + g(\theta) \cdot w }.
$$
%In particular,
%$$
%\Prob{ \lambda_{\max}( \mtx{Y}_n ) \geq t \quad\text{and}\quad
%\lambda_{\max}( \mtx{V}_{n-1} ) \leq v }
%	\leq d \cdot \inf_{\theta > 0} \econst^{-\theta t + g(\theta) v }.
%$$
\end{thm}

\begin{proof}
To begin, note that the cgf hypothesis~\eqref{eqn:cgf-hyp-adapted} holds in the presence of~\eqref{eqn:mgf-hyp-adapted} because the logarithm is an operator monotone function~\cite[Ch.~V]{Bha97:Matrix-Analysis}.

The overall proof strategy is identical with the stopping-time technique used by Freedman~\cite{Fre75:Tail-Probabilities}.  
Fix a positive parameter $\theta$, which we will optimize later.  Following the discussion in~\S\ref{sec:supermartingale}, we introduce the random process $S_k := G_{\theta}(\mtx{Y}_k, \mtx{W}_{k})$.  Lemma~\ref{lem:Sk-supermartingale} implies that $\{ S_k \}$ is a positive supermartingale with initial value $d$.  These simple properties of the auxiliary random process distill all the essential information from the hypotheses of the theorem.

Define a stopping time $\kappa$ by finding the first time instant $k$ when the maximum eigenvalue of the partial sum process reaches the level $t$ even though the sum of cgf bounds has maximum eigenvalue no larger than $w$.
$$
\kappa := \inf\{ k \geq 0 : \lambda_{\max}( \mtx{Y}_k ) \geq t 
	\ \text{ and }\ \lambda_{\max}(\mtx{W}_{k}) \leq w\}.
$$
When the infimum is empty, the stopping time $\kappa = \infty$.  Consider a system of exceptional events:
$$
E_k := \{ \lambda_{\max}(\mtx{Y}_k) \geq t \ \text{ and }\ 
	\lambda_{\max}(\mtx{W}_{k}) \leq w \}
	\quad\text{for $k = 0, 1, 2, \dots$.}
$$
Construct the event $E := \bigcup\nolimits_{k=0}^\infty E_k$ that one or more of these exceptional situations takes place.  The intuition behind this definition is that the partial sum $\mtx{Y}_k$ is typically not large unless the process $\{\mtx{X}_k\}$ has varied substantially, a situation that the bound on $\mtx{W}_k$ disallows.  As a result, the event $E$ is rather unlikely.

%We put $\kappa = n$ when the set is empty.  Observe that the stopping time $\kappa \leq n$ on the event $E$.

%Next, we must control the minimum value of the stopped process $S_\kappa$ on the event $E$.  %Conditional on $E$,
% it holds that
%$$
%\lambda_{\max}(\mtx{Y}_\kappa) \geq t
%\quad\text{and}\quad
%\lambda_{\max}(\mtx{W}_{\kappa}) \leq w.
%$$
%Therefore,

%Lemma~\ref{lem:supermartingale} demonstrates that the process $S_k = F_{\theta}(\mtx{Y}_k, \mtx{V}_{k-1})$ is a bounded, nonnegative supermartingale.  

We are prepared to estimate the probability of the exceptional event.  First, note that $\kappa < \infty$ on the event $E$.  Therefore, Lemma~\ref{lem:G-monotone} provides a conditional lower bound for the process $\{S_k\}$ at the stopping time $\kappa$:
$$
S_{\kappa} = G_{\theta}(\mtx{Y}_{\kappa}, \mtx{W}_{\kappa})
	\geq \econst^{ \theta t - g(\theta) \cdot w }
\quad\text{on the event $E$.}
$$
Since $\Expect S_k \leq d$ for each (finite) index $k$,
\begin{multline*}
d \geq\sum\nolimits_{k=1}^\infty \Expect[ S_\kappa \, | \, \kappa = k ]
		\cdot \Prob{ \kappa = k }
	= \Expect[ S_{\kappa} \, | \, \kappa < \infty ]
	\geq \int_{\{ \kappa < \infty\}} S_{\kappa} \idiff{\mathbb{P}} \\
	\geq \int_{E} S_{\kappa} \idiff{\mathbb{P}}
	\geq \Probe{E} \cdot \inf\nolimits_{E} S_\kappa
	\geq \Probe{E} \cdot \econst^{\theta t - g(\theta)\cdot w}.
\end{multline*}
We require the fact that $S_{\kappa}$ is positive to justify these inequalities.  Rearrange the relation to obtain
$$
\Probe{E} \leq d \cdot \econst^{-\theta t + g(\theta)\cdot w}.
$$
Minimize the right-hand side with respect to $\theta$ to complete the main part of the argument.
\end{proof}

\section{Proof of Freedman's Inequality} \label{sec:freedman}

In this section, we use the general martingale deviation bound, Theorem~\ref{thm:main-result-adapted}, to prove a stronger version of Theorem~\ref{thm:matrix-freedman}.

\begin{thm} \label{thm:adapted-bennett}
Consider an adapted sequence $\{ \mtx{X}_k \}$ of self-adjoint matrices with dimension $d$ that satisfy the relations
$$
\Expect_{k-1} \mtx{X}_k = \mtx{0}
\quad\text{and}\quad
\lambda_{\max}( \mtx{X}_k ) \leq R
\quad\text{almost surely}
\quad\text{for $k = 1, 2, 3, \dots$.}
$$
Define the partial sums
$$
\mtx{Y}_k := \sum\nolimits_{j=1}^k \mtx{X}_j
\quad\text{and}\quad
\mtx{W}_k := \sum\nolimits_{j=1}^k \Expect_{j-1} \big(\mtx{X}_j^2\big)
\quad\text{for $k = 0, 1, 2, \dots$}.
$$
Then, for all $t \geq 0$ and $\sigma^2 > 0$,
$$
\Prob{ \exists k \geq 0 : \lambda_{\max}(\mtx{Y}_k) \geq t \ \text{ and }\ 
	\norm{\mtx{W}_k} \leq \sigma^2 }
	\leq d \cdot \exp \left\{ - \frac{\sigma^2}{R^2} \cdot h\!\left( \frac{Rt}{\sigma^2} \right) \right\}.
$$
The function $h(u) := (1+u)\log(1+u) - u$ for $u \geq 0$.
\end{thm}

Theorem~\ref{thm:matrix-freedman} follows easily from this result.

\begin{proof}[Proof of Theorem~\ref{thm:matrix-freedman} from Theorem~\ref{thm:adapted-bennett}]
To derive Theorem~\ref{thm:matrix-freedman}, we note that the difference sequence of $\{\mtx{X}_k\}$ a matrix martingale $\{\mtx{Y}_k\}$ satisfies the conditions of Theorem~\ref{thm:adapted-bennett} and the martingale can be expressed using partial sums of the difference sequence.  Finally, we apply the numerical inequality
$$
h(u) \geq \frac{u^2/2}{1 + u/3}
\quad\text{for $u \geq 0$},
$$
which we obtain by comparing derivatives.
\end{proof}

\subsection{Demonstration of Theorem~\ref{thm:adapted-bennett}}

We conclude with the proof of Theorem~\ref{thm:adapted-bennett}.  The argument depends on the following estimate for the moment generating function of a zero-mean random matrix whose eigenvalues are uniformly bounded.  See~\cite[Lem.~6.7]{Tro10:User-Friendly-arxiv} for the proof.

\begin{lemma}[Freedman mgf] \label{lem:bernstein-bdd-mgf}
Suppose that $\mtx{X}$ is a random self-adjoint matrix that satisfies
$$
\Expect \mtx{X} = \mtx{0}
\quad\text{and}\quad
\lambda_{\max}(\mtx{X}) \leq 1.
%\quad\text{almost surely}.
$$
Then
$$
\Expect \econst^{\theta \mtx{X}}
	\psdle \exp\left( (\econst^{\theta} - \theta - 1)
	\cdot \Expect(\mtx{X}^2) \right)
\quad\text{for $\theta > 0$.}
$$
\end{lemma}

The main result follows quickly from this lemma.

\begin{proof}[Proof of Theorem~\ref{thm:adapted-bennett}]
We assume that $R = 1$; the general result follows by re-scaling since $\mtx{Y}_k$ is 1-homogeneous and $\mtx{W}_k$ is 2-homogeneous.  Invoke Lemma~\ref{lem:bernstein-bdd-mgf} conditionally to see that
$$
\Expect_{k-1} \econst^{\theta \mtx{X}_k}
	\psdle \exp\left( g(\theta) \cdot \Expect_{k-1} \big(\mtx{X}_k^2 \big) \right)
\quad\text{where $g(\theta) := \econst^{\theta} - \theta - 1$}.
$$
Theorem~\ref{thm:main-result-adapted} now implies that
$$
\Prob{ \exists k \geq 0 : \lambda_{\max}(\mtx{Y}_k) \geq t \ \text{ and }\ 
	\lambda_{\max}(\mtx{W}_k) \leq \sigma^2 }
	\leq d \cdot \inf_{\theta > 0} \econst^{-\theta t + g(\theta) \cdot \sigma^2}.
$$
The infimum is achieved when $\theta = \log(1 + t/\sigma^2)$.  Finally, note that the norm of a positive-semidefinite matrix, such as $\mtx{W}_k$, equals its largest eigenvalue.
%Substitute and simplify to establish the result.
\end{proof}

%\notate{Prove the rectangular version.}

\section*{Acknowledgments}

Roberto Oliveira introduced me to Freedman's inequality and encouraged me to apply the methods from~\cite{Tro10:User-Friendly-arxiv} to study the matrix extension of Freedman's result.  I would also like to thank Yao-Liang Yu, who pointed out an inconsistency in the proof of Theorem~\ref{thm:main-result-adapted} and who proposed the argument in Lemma~\ref{lem:bernstein-bdd-mgf}.  Richard Chen and Alex Gittens have helped me root out (numerous) typographic errors.

\bibliographystyle{alpha}
\bibliography{user-friendly}

\end{document}

%% file: macro-file.tex
\usepackage{amsmath}
\usepackage{amssymb}
\usepackage{bm}
\usepackage{mathrsfs}
\usepackage[dvips]{graphicx}

\usepackage{color}

\usepackage{url}

\usepackage{supertabular}

\usepackage{alg}

\makeatletter
\def\algspacing{\alg@unmargin}
\makeatother

\newtheorem{thm}{Theorem}
\newtheorem{cor}[thm]{Corollary}
\newtheorem{lemma}[thm]{Lemma}

\theoremstyle{remark}

\numberwithin{equation}{section}

\numberwithin{thm}{section}

\DeclareMathAlphabet{\mathsfsl}{OT1}{cmss}{m}{sl}

\newcommand{\term}{\emph}

	{\makebox{\phantom{}} \\ %
		\textsc{Input:}\begin{itemize}}%
	{\end{itemize}}

	{\textsc{Output:}\begin{itemize}}%
	{\end{itemize}}

	{\textsc{Procedure:}\begin{enumerate}}%
	{\end{enumerate}}

\renewcommand{\phi}{\varphi}

\newcommand{\eps}{\varepsilon}

\newcommand{\econst}{\mathrm{e}}

\newcommand{\onevct}{\mathbf{e}}

\newcommand{\Id}{\mathbf{I}}

\newcommand{\coll}[1]{\mathscr{#1}}

\newcommand{\abs}[1]{\left\vert {#1} \right\vert}

\newcommand{\diff}[1]{\mathrm{d}{#1}}
\newcommand{\idiff}[1]{\, \diff{#1}}

\newcommand{\Probe}[1]{\mathbb{P}\left({#1}\right)}
\newcommand{\Prob}[1]{\mathbb{P}\left\{ {#1} \right\}}
\newcommand{\Expect}{\operatorname{\mathbb{E}}}

\newcommand{\vct}[1]{\bm{#1}}
\newcommand{\mtx}[1]{\bm{#1}}

\newcommand{\adj}{*}

\newcommand{\trace}{\operatorname{tr}}

\newcommand{\psdle}{\preccurlyeq}

\newcommand{\norm}[1]{\left\Vert {#1} \right\Vert}

\newcommand{\enorm}[1]{\norm{#1}_2}
\newcommand{\enormsq}[1]{\enorm{#1}^2}